\documentclass[journal]{IEEEtran}
\AtBeginDvi{} 
\usepackage{amssymb}
\usepackage{framed}
\usepackage{bm}
\usepackage{cite}
\usepackage{comment}
\usepackage{algorithm}
\usepackage{algorithmicx}
\usepackage{algpseudocode}
\usepackage{enumitem}
\usepackage{amsmath}
\usepackage{amsfonts}
\usepackage{color}
\usepackage{cases}
\newtheorem{theorem}{Theorem}

\newtheorem{lemma}{Lemma}
\newtheorem{corollary}{Corollary}

\newtheorem{proposition}{Proposition}

\newtheorem{proof}{Proof}

\usepackage{blkarray}

\usepackage{caption}
\usepackage[subrefformat=parens]{subcaption}
\newcommand{\Real}{\mathbb{R}}

\DeclareMathOperator{\rank}{rank}

\usepackage{latexsym}
\def\qed{\hfill $\Box$} 

\usepackage{cite}

\ifCLASSINFOpdf
  \usepackage[pdftex]{graphicx}
\else
  \usepackage[dvipdfmx]{graphicx}
\fi
\usepackage{amsmath}
\usepackage{url}

\begin{document}
%
\title{Minimal controllability problem on linear structural descriptor systems with forbidden nodes}
\author{Shun Terasaki and Kazuhiro Sato \thanks{S. Terasaki and K. Sato are with the Department of Mathematical Informatics, Graduate School of Information Science and Technology, The University of Tokyo, Tokyo 113-8656, Japan, email:terasaki@g.ecc.u-tokyo.ac.jp (S. Terasaki), kazuhiro@mist.i.u-tokyo.ac.jp (K. Sato) }}
\maketitle
\thispagestyle{empty}
\pagestyle{empty}

\begin{abstract}
We consider a minimal controllability problem (MCP), which determines the minimum number of input nodes for a descriptor system to be structurally controllable.
We investigate the ``forbidden nodes'' in descriptor systems, denoting nodes that are unable to establish connections with input components.
The three main results of this work are as follows.
First, we show a solvability condition for the MCP with forbidden nodes using graph theory such as a bipartite graph and its Dulmage--Mendelsohn decomposition.
Next, we derive the optimal value of the MCP with forbidden nodes. The optimal value is determined by an optimal solution for constrained maximum matching, and this result includes that of the standard MCP in the previous work. 
Finally, we provide an efficient algorithm for solving the MCP with forbidden nodes based on an alternating path algorithm.
\end{abstract}


%
\IEEEpeerreviewmaketitle

\section{Introduction} \label{sec:intro}
%
%
%


Controllability analysis for large-scale network systems, such as multi-agent systems\cite{mesbahi2010graph, rahmani2009controllability}, brain networks\cite{gu2015controllability, muldoon2016stimulation}, and power networks\cite{pagani2013power, yang2020critical} has received a great deal of interest in recent years\textcolor{black}{, because it can be used to find important nodes \cite{liu2011controllability}.}
Controllability analysis problems include:
\begin{itemize}
    \item quantitative problems; the maximization problems of controllability metrics \cite{pasqualetti2014controllability, summers2015submodularity, clark2017submodularity, romao2018distributed, sato2020controllability}.
    \item qualitative problems; selecting input problems that render the system controllable \cite{liu2011controllability, olshevsky2015minimal, pequito2015framework, clark2017input, terasaki2021minimal}.
\end{itemize}
The quantitative problems require the system parameters, which are not precisely determined in practical systems. In addition, quantitative problems often become computationally intractable when the state dimension becomes large.
Conversely, the structural information of a system, \textit{i.e.}, the nonzero patterns of system parameters, is usually known.
This is an advantage for qualitative problems that deal only with nonzero patterns.
Also, it is known that the structural controllability \cite{lin1974structural} of a structural system can be checked efficiently using graph algorithms.
Thus, for large-scale network systems, it is more appropriate to consider qualitative problems. 

Therefore, we consider a Minimal Controllability Problem (MCP) of the following structural descriptor network system with state $x(t) \in \Real^n$ and input $u(t) \in \Real^m$\textcolor{black}{, where $\Real$ is the set of real numbers}:
\begin{align}
    F\dot{x}(t) = Ax(t) + Bu(t),  \label{eq:descriptor}
\end{align}
where $F, A\in\Real^{n\times n}, B\in\Real^{n\times m}$, and $F\in \Real^{n\times n}$ can be a singular matrix. 
That is, we assume that although the specific elements of $F$, $A$, and $B$ remain unknown, their nonzero patterns are known.
The descriptor formulation aptly models practical systems featuring algebraic constraints, such as those found in electric circuit systems \cite{dai1989singular,murota00,duan2010analysis}.
MCPs can be divided into two main problems\cite{liu2016control}:
\begin{itemize}
    \item MCP0: A problem that finds an $n\times m$ matrix $B$ for system (\ref{eq:descriptor}) to be structurally controllable where $m$ is minimum.
    For $F=I_n$ or $F\neq I_n$, efficient algorithms exist for solving MCP0 \cite{liu2011controllability, olshevsky2015minimal, pequito2015framework, terasaki2021minimal}.
    
    \item MCP1: A problem that finds an $n\times n$ diagonal matrix $B$ for system (\ref{eq:descriptor}) to be structurally controllable and the number of nonzero elements in $B$ is minimum \cite{pequito2015framework}.
   \textcolor{black}{Although a polynomial time algorithm exists \cite{clark2017input} for a special case of \eqref{eq:descriptor} with $F\neq I_n$,
    MCP1 is known to be NP-hard \cite{terasaki2021minimal} in general.}
\end{itemize}
It should be noted that as mentioned in Section 3.2 in \cite{RAMOS2022110229},
there are only a few papers on MCP0 or MCP1 for system \eqref{eq:descriptor} with $F\neq I_n$
although just structural controllability analysis under the assumption of a given $(F,A,B)$ has been studied in \cite{murota1984descriptor, murota00, reinschke1997digraph}.

The previous work in \cite{terasaki2021minimal} on structural descriptor system (\ref{eq:descriptor}) has not considered constraints on the input destination. There is a gap between practical situations since most physical systems have state variables to which inputs cannot be directly connected. 
For instance, consider a system in which the position $x(t)$ and velocity $v(t)$ of an object are the state variables. In this case, the state equation involves $\dot{x}(t)=v(t)$, but it does not make practical sense to add an input to this equation.
Thus, specifying forbidden targets that cannot be connected to inputs is an important practical constraint.
Therefore, \cite{olshevsky2015minimal} introduced forbidden nodes to MCP1 for system (\ref{eq:descriptor}) with $F=I_n$.
However, no work applies MCPs with $F\ne I_n$. 

In this paper, we address MCP0 with forbidden nodes for structural descriptor system (\ref{eq:descriptor}), because, in general, MCP1 for system (\ref{eq:descriptor}) with $F\ne I_n$ is NP-hard, as shown in \cite{terasaki2021minimal}.
Here, the forbidden nodes correspond to the indices of ``equations'', while those in \cite{olshevsky2015minimal}, which studied (\ref{eq:descriptor}) with $F=I_n$, correspond to the indices of ``variables''. 
This naturally generalizes to $F \ne I_n$.
In fact, 
\begin{itemize}
    \item for $F=I_n$, the time evolution of $x_i$ is characterized by the $i$-th equation of (\ref{eq:descriptor}).
    Thus, in this case, we can regard the index of equations as that of variables.
    \item for $F\ne I_n$, the time evolution of $x_i$ is not characterized by the $i$-th equation of (\ref{eq:descriptor}).
    Thus, in contrast to the case of $F=I_n$, here we cannot regard the index of equations as that of variables.
\end{itemize}

The contributions of this study can be summarized as follows.
\begin{itemize}
    \item We show a necessary and sufficient condition for the existence of the optimal solution of MCP0 with forbidden equations for structural descriptor system (\ref{eq:descriptor}), which is described in the language of graph theory. This result is also useful in constructing the optimal solution.
    \item We provide the optimal value of MCP0 with forbidden equations for structural descriptor system (\ref{eq:descriptor}), by employing the graph-theoretic properties of the system, such as a bipartite graph and its Dulmage--Mendelsohn (DM) decomposition. The optimal value shows that the minimum number of input nodes is determined by a variant of the maximum matching problem with constraints on the matched nodes. This result includes that of the MCP0 without forbidden equations for descriptor system (\ref{eq:descriptor}) \cite{terasaki2021minimal}. 
    \item We also provide an efficient algorithm for solving the above special matching problem by using the alternating path algorithm. The time complexity of this algorithm is $O(|V|+|E|\sqrt{|V|})$, which is on par with the algorithm for the MCP0 \cite{liu2011controllability, terasaki2021minimal} without forbidden equations, where $|V|$ and $|E|$ are the numbers of nodes and edges of the bipartite graph corresponding to descriptor system (\ref{eq:descriptor}), respectively.
\end{itemize}

The remainder of this paper is organized as follows.
The basic concepts of graph theory are summarized in Section~\ref{sec:graph}.
The formulation of MCP0 with forbidden equations for structural descriptor systems is described in Section~\ref{sec:prb}.
In Section~\ref{sec:analysis}, we provide the analysis and the algorithm of MCP0 with forbidden equations for the descriptor system (\ref{eq:descriptor}). The conclusions are presented in Section~\ref{sec:conclusion}.

\label{sec:prevresults}
\section{Basic concepts of graph theory}
\label{sec:graph}
In this section, we present a comprehensive overview of the fundamental concepts of graph theory that are used in this paper.

A strongly connected component (SCC) of a directed graph with node set $V$ is a maximal subset $C\subseteq V$ whose nodes $u,v \in C$ can be connected by a directed path on the graph. 

\textcolor{black}{Let $G = (V^+,V^-;E)$ be a bipartite graph.
For an edge $e = (v^+,v^-) \in E$, $\partial^+ e$ and $\partial^- e$ denote the 
 nodes of $v^+\in V^+$ and $v-\in V^-$, respectively.
 That is, $\partial^+: E\rightarrow V^+$ and $\partial^-: E\rightarrow V^-$. 
The edge set $M\subseteq E$ is a matching if it does not share nodes of each edge.} A matching $M$ is termed maximum matching if $M$ contains the largest possible number of edges.
We define symbol $\nu(G)$ as the size of a maximum matching of $G$.


We introduce the DM decomposition for a bipartite graph $G = (V^+,V^-;E)$, which is the unique decomposition algorithm for bipartite graphs (see \cite{murota00} for details). Algorithm \ref{alg:dmdecomp} describes DM decomposition.
We define $M$ as a maximum matching of $G$ and an auxiliary directed graph $\tilde{G}_M$. The edges of $\tilde{G}_M$ are oriented from $V^+$ to $V^-$ except for $M$.
The decomposition is illustrated in Fig.~\ref{fig:dmdecompalgorithm}.
\begin{figure}[!t]
  \begin{algorithm}[H]
    \caption{Algorithm for DM decomposition.}
    \label{alg:dmdecomp}
    \begin{algorithmic}[1]
    \State Find a maximum matching $M$ on $G$.
    \State Construct an auxiliary directed graph $\tilde{G}_M$.
    \State $V_0 := \{ v\ \in V^+\cup V^- \mid \exists u\in V^+\setminus \partial^+ M \ u \to_{\tilde{G}_M} v \}$, where $u \to_{\tilde{G}_M} v$ indicates the existence of a directed path on $\tilde{G}_M$ from $u$ to $v$.
    \State $V_\infty := \{ v\ \in V^+\cup V^- \mid \exists u\in  V^-\setminus \partial^- M \ v\to_{\tilde{G}_M} u \}$.
    \State \textcolor{black}{Let $G'$ be the subgraph of $\tilde{G}_M$ defined by deleting all nodes of $V_0\cup V_\infty$ and all edges adjacent to the nodes.}
    \State Let $V_k\ (k=1,\dots,b)$ be the SCCs of $G'$. Let undirected graph $G_k\ (k=0,1,\dots,b,\infty)$ be the subgraph of $G$ induced on $V_k\ (k=0,1,\dots,b,\infty)$.
    \end{algorithmic}
  \end{algorithm}
\end{figure}
Subgraphs $G_k \ ( k = 0 , \dots , b , \infty )$ in Step 6 of Algorithm \ref{alg:dmdecomp} are called DM components of $G$. $G_k \ ( k = 1,\dots,b )$ are called consistent DM components; $G_0$ and $G_\infty$ are called inconsistent DM components.
The order $G_i \preceq G_j$ for the consistent DM components $G_i$ and $G_j$ is defined as
\begin{center}
     \text{``There is a directed path on $\tilde{G}_M$ from $G_j$ to $G_i$,"}
\end{center}
and the order between the consistent DM component $G_i$ and inconsistent DM components $G_0$ and $G_\infty$ are defined as $G_0 \preceq G_i$ and $G_i \preceq G_\infty$, respectively. Then, $\preceq$ is a partial order.
Moreover, the decomposition constructed by Algorithm \ref{alg:dmdecomp} does not depend on an initially chosen maximum matching $M$ of $G$ in step 1), as shown in Lemma~2.3.35 in \cite{murota00}.
For example, in Fig.~\ref{fig:dmdecompalgorithm}, there is a directed edge from $G_2$ to $G_1$, and thus
\begin{align*}
    G_0 \preceq G_1 \preceq G_2 \preceq G_\infty.
\end{align*}

The greatest computational bottleneck in the construction of the DM decomposition is to find the maximum matching of $G$. This can be achieved in $O(|E|\sqrt{|V|})$ by using the augmentation path algorithm \cite{korte2012combinatorial}. Thus, the computational complexity of DM decomposition is $O(|E|\sqrt{|V|})$.

\begin{figure}
    \centering
    \includegraphics[width=8.5cm]{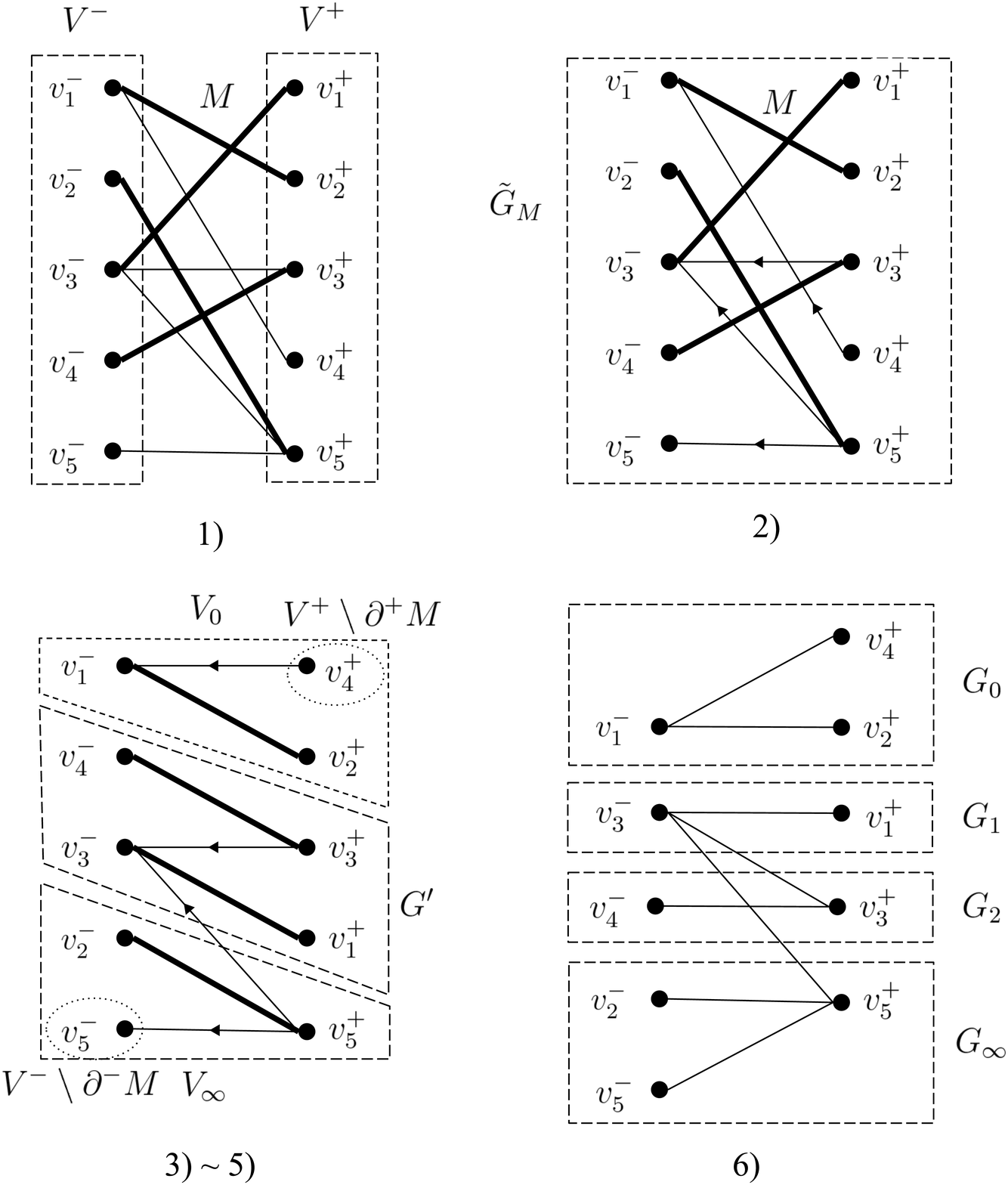}
    \caption{Construction of DM decomposition. The bold edges represent maximum matching $M$.}
    \label{fig:dmdecompalgorithm}
\end{figure}

\section{Problem Settings}
\label{sec:prb}
In this section, we formulate MCP0 with forbidden equations for the descriptor system (\ref{eq:descriptor}).

First, we assume that system (\ref{eq:descriptor}) is solvable, \textit{i.e.}, for any initial state $x(0)$ with an admissible input $u(t)$, there exists a unique solution $x(t)$ to Eq.~(\ref{eq:descriptor}). This condition is equivalent to 
\begin{align}
    \rank (A-sF) = n, \label{eq:solvability}
\end{align}
where $s$ is an indeterminant \cite{murota00} \cite{yip1981solvability}.

In this paper, we call system (\ref{eq:descriptor}) controllable if for any admissible initial state $x(0)$ that satisfies equation (\ref{eq:descriptor}), there exists an input $u(t)$ and a final time $T\geq 0$ such that $x(T) = 0$. This controllability is usually referred to as R-controllability. Additional controllability concepts can be found within the framework of descriptor system theory \cite{berger2013controllability}.

An algebraic characterization of controllability can be described as follows \cite{yip1981solvability, berger2013controllability}:
\begin{proposition}
    \label{prop:scalg}
    Descriptor system (\ref{eq:descriptor}) is controllable \textit{if and only if} 
    \begin{align}
        \rank \left[ A-zF \mid B \right] & = n, \label{eq:rank}
    \end{align}
    \textcolor{black}{where $z$ is any complex number.}
\end{proposition}

System (\ref{eq:descriptor}) is termed structurally controllable if condition (\ref{eq:rank}) in Proposition \ref{prop:scalg} holds for (\ref{eq:descriptor}) with generic matrices $F$, $A$, and $B$.
It should be noted that a matrix is considered generic if each nonzero element is an independent parameter. For a more precise definition of the generic matrix, see \cite{murota00}.

We now introduce MCP0 with forbidden equations for descriptor system (\ref{eq:descriptor}). Let $R = \{e_1, \dots, e_n\}$ be a set of equation indices and $\mathcal{F}$ be a subset of $R$ that denotes forbidden equations. Then, MCP0 with forbidden equations for descriptor system (\ref{eq:descriptor}) can be formalized as
{\small
\begin{numcases}{}
        \underset{\text{$B\in \mathcal{G}^{n\times m}$}}{\text{minimize}} & $m$ \label{prob:cmcp0} \\
        \text{subject to} & I) system (\ref{eq:descriptor}) is structurally controllable, \notag \\
        & II) indices of nonzero rows of $B$ are in $R\setminus \mathcal{F}$, \notag
\end{numcases}
}
\noindent where $\mathcal{G}^{n\times m}$ denote the set of all $n\times m$ generic matrices.
The significant difference between the standard MCP0 and Problem (\ref{prob:cmcp0}) is II) in (\ref{prob:cmcp0}). This constraint limits the input destination.

For instance, consider descriptor system (\ref{eq:descriptor}) with
\begin{align}
    \label{eq:examplesystem}
    F = \begin{bmatrix}
        0 & f_1  & 0 & 0 & f_2 \\
        f_3 & 0 & 0 & 0 & 0 \\
        0 & f_4 & 0 & 0 & 0 \\
        0 & 0 & 0 & f_5 & 0 \\
        0 & 0 & 0 & 0 & 0
    \end{bmatrix},\ 
    A  = \begin{bmatrix}
        a_1 & 0 & 0 & 0 & 0 \\
        a_2 & 0 & 0 & 0 & 0 \\
        a_3 & 0 & 0 & 0 & 0 \\
        0 & 0 & a_4 & 0 & a_5 \\
        0 & 0 & a_6 & a_7 & 0
    \end{bmatrix}.
\end{align}
Then, $R = \{e_1,e_2,\dots, e_5\}$.
Let $\mathcal{F} = \{e_3, e_4\}$ be the set of indices of forbidden equations. 
The matrices
\begin{align}
    B_1 = \begin{bmatrix}
        b_1 & 0 & 0 & 0 & 0
    \end{bmatrix}^\top, \ 
    B_2 = \begin{bmatrix}
        0 & 0 & 0 & 0 & b_2\\
        0 & b_1 & 0 & 0 & 0
    \end{bmatrix}^\top
    \label{eq:exampleinputs}
\end{align}
satisfy condition II).
\section{Analysis and Algorithm}
\label{sec:analysis}
In this section, we provide the analysis and the algorithm of MCP0 with forbidden equations for descriptor system (\ref{eq:descriptor}) with generic matrices using a graph-theoretic approach.

To this end, we first describe the graph representation of the descriptor system (\ref{eq:descriptor}) with generic matrices and graph-theoretical controllability.
Although there are several graph representations for descriptor system (\ref{eq:descriptor}) \cite{reinschke1997digraph, yamada1985generic, murota00}, we used the bipartite graph representation \cite{murota00} in this study.
While the directed graph representation is common for system (\ref{eq:descriptor}) with $F=I_n$, the bipartite graph representation is often used for $F\ne I_n$ \cite{murota1984descriptor}.

The bipartite graph $G = (V^+,V^-; E)$ associated with descriptor system (\ref{eq:descriptor}) is defined as follows: the node sets $V^+$ and $V^-$ are defined as 
\begin{align*}
\begin{cases}
   V^+ := X \cup U, \\
   V^- := \{e_1,\dots,e_n\},
\end{cases}
\end{align*}
where the state node set $X$ and the input node set $U$ are defined as 
\begin{align*}
    X := \{x_1,\dots,x_n\},\quad U := \{u_1,\dots,u_m\}, 
\end{align*}
respectively, and $e_i$ in $V^-$ corresponds to the $i$-th equation of system (\ref{eq:descriptor}). That is, $V^+$ consists of state variables and inputs, and $V^-$ is the set of indices of equations.
Then, the edge set $E$ is defined as
\begin{align}
   E := E_A \cup E_F \cup E_B,
\label{eq:defgraph}
\end{align}
with $E_A := \{ ( e_i , x_j ) \mid A_{ij} \neq 0 \}$, $E_F := \{ ( e_i , x_j ) \mid F_{ij} \neq 0 \}$ and $E_B = \{ ( e_i , u_j ) \mid B_{ij} \neq 0 \}$.
An edge belonging to $E_F$ is termed an s-arc.
We also define important subgraphs of $G$ as $G_{A-sF} = (X,V^-; E_A\cup E_F)$, $G_{\left[A \mid B\right]} = (V^+,V^-; E_A\cup E_B)$, and $G_A = (X,V^-;E_A)$.
For example, the bipartite representation of descriptor system (\ref{eq:descriptor}) with (\ref{eq:examplesystem}) and $B_2$ in (\ref{eq:exampleinputs}), and its subgraphs are illustrated in Fig.~\ref{fig:example}.

\begin{figure*}[t]
   \begin{tabular}{cccc}
     \begin{minipage}[t]{0.24\hsize}
       \centering
       \includegraphics[keepaspectratio, scale=0.3]{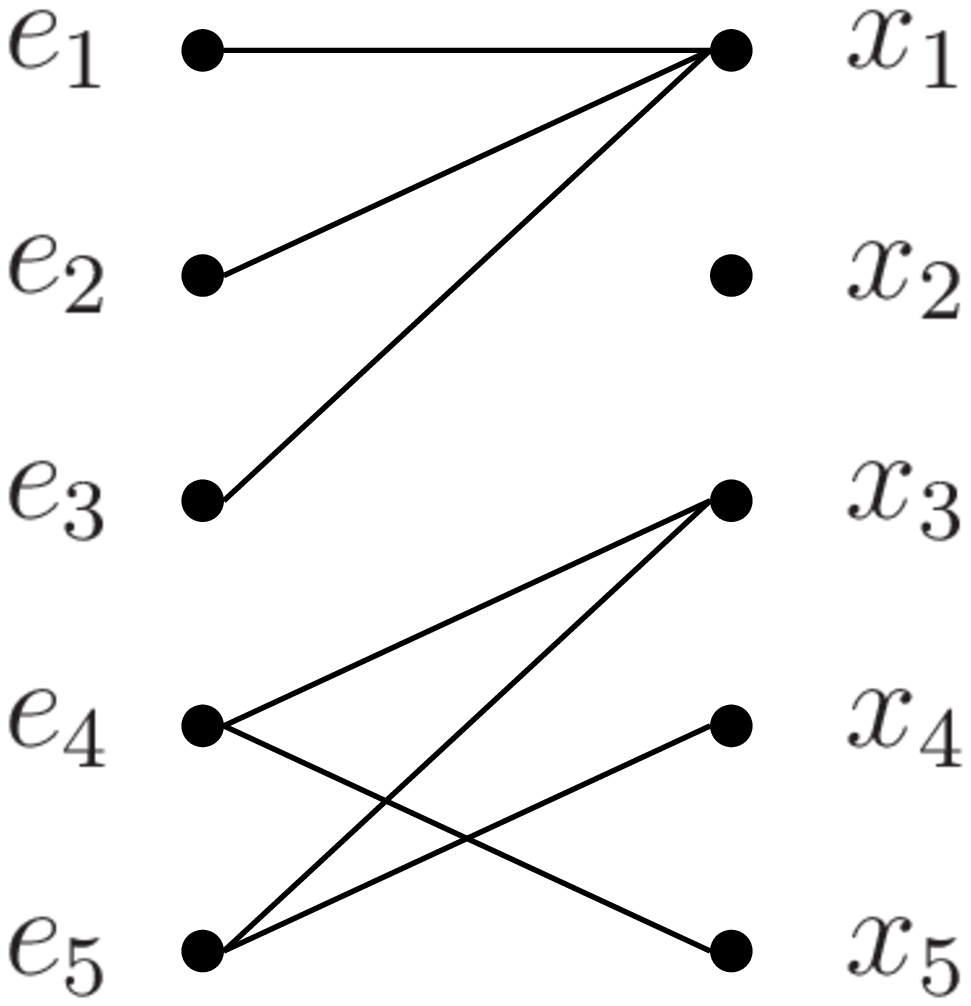}
       \subcaption{$G_A$}
     \end{minipage} &
     \begin{minipage}[t]{0.24\hsize}
       \centering
       \includegraphics[keepaspectratio, scale=0.3]{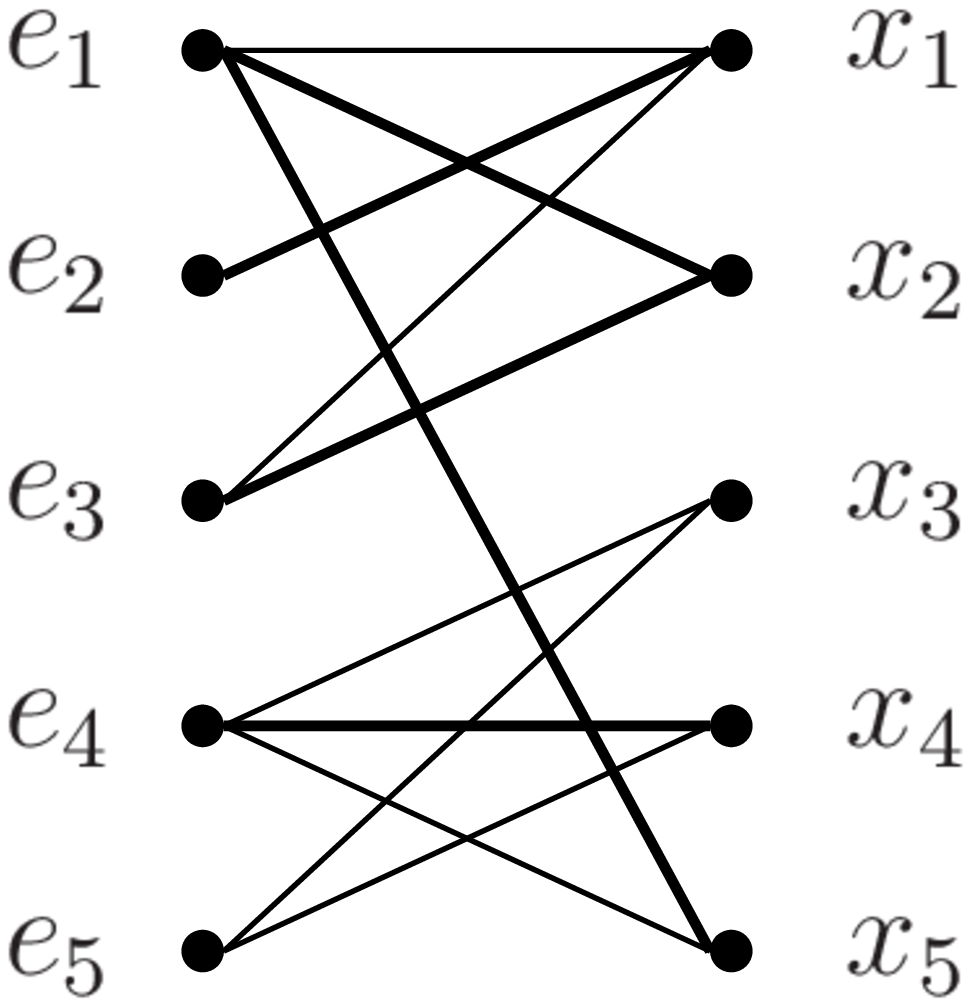}
       \subcaption{$G_{A-sF}$}
     \end{minipage}
     \begin{minipage}[t]{0.24\hsize}
       \centering
       \includegraphics[keepaspectratio, scale=0.3]{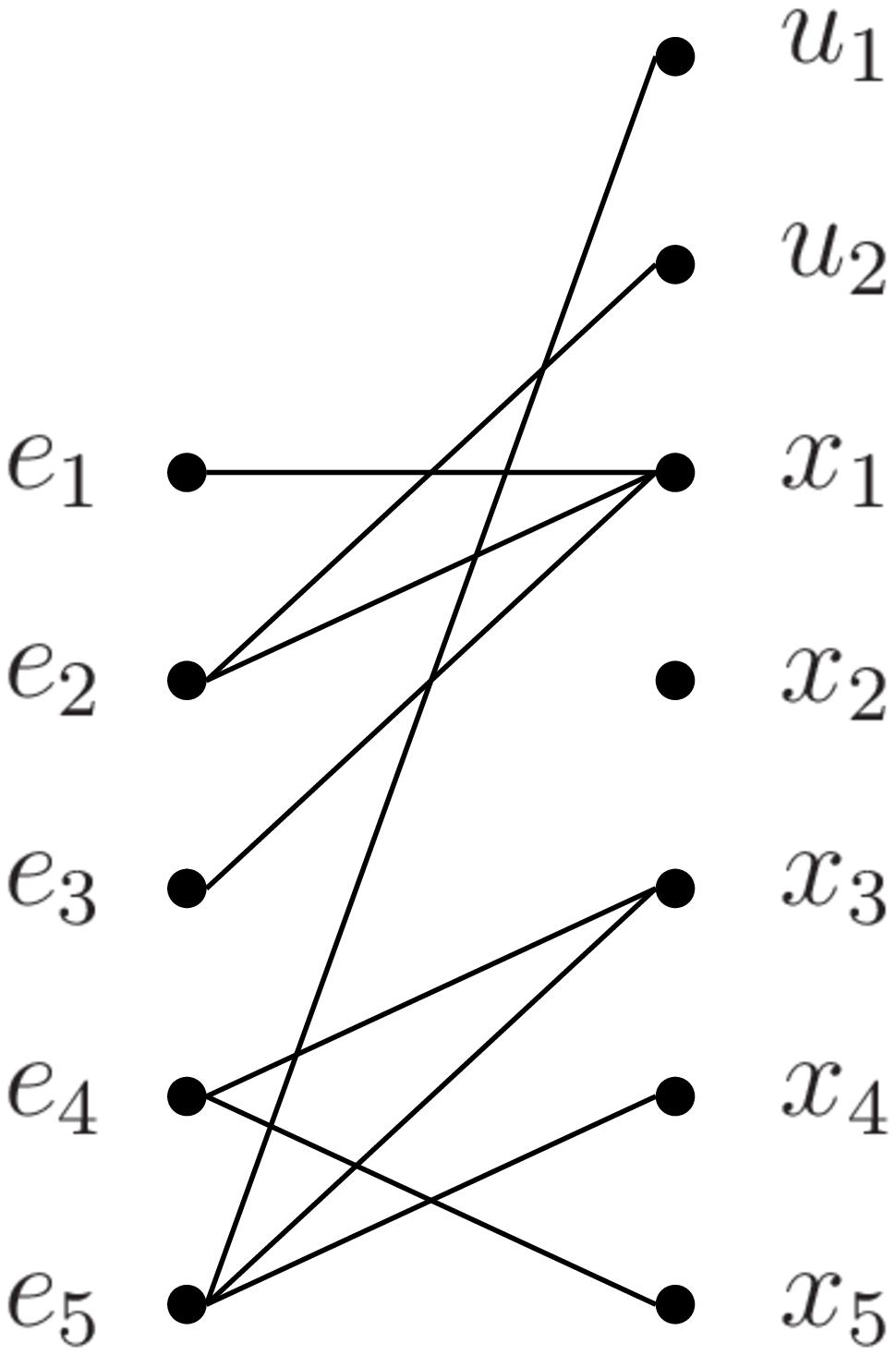}
       \subcaption{$G_{\left[A\mid B\right]}$}
     \end{minipage} &
     \begin{minipage}[t]{0.24\hsize}
       \centering
       \includegraphics[keepaspectratio, scale=0.3]{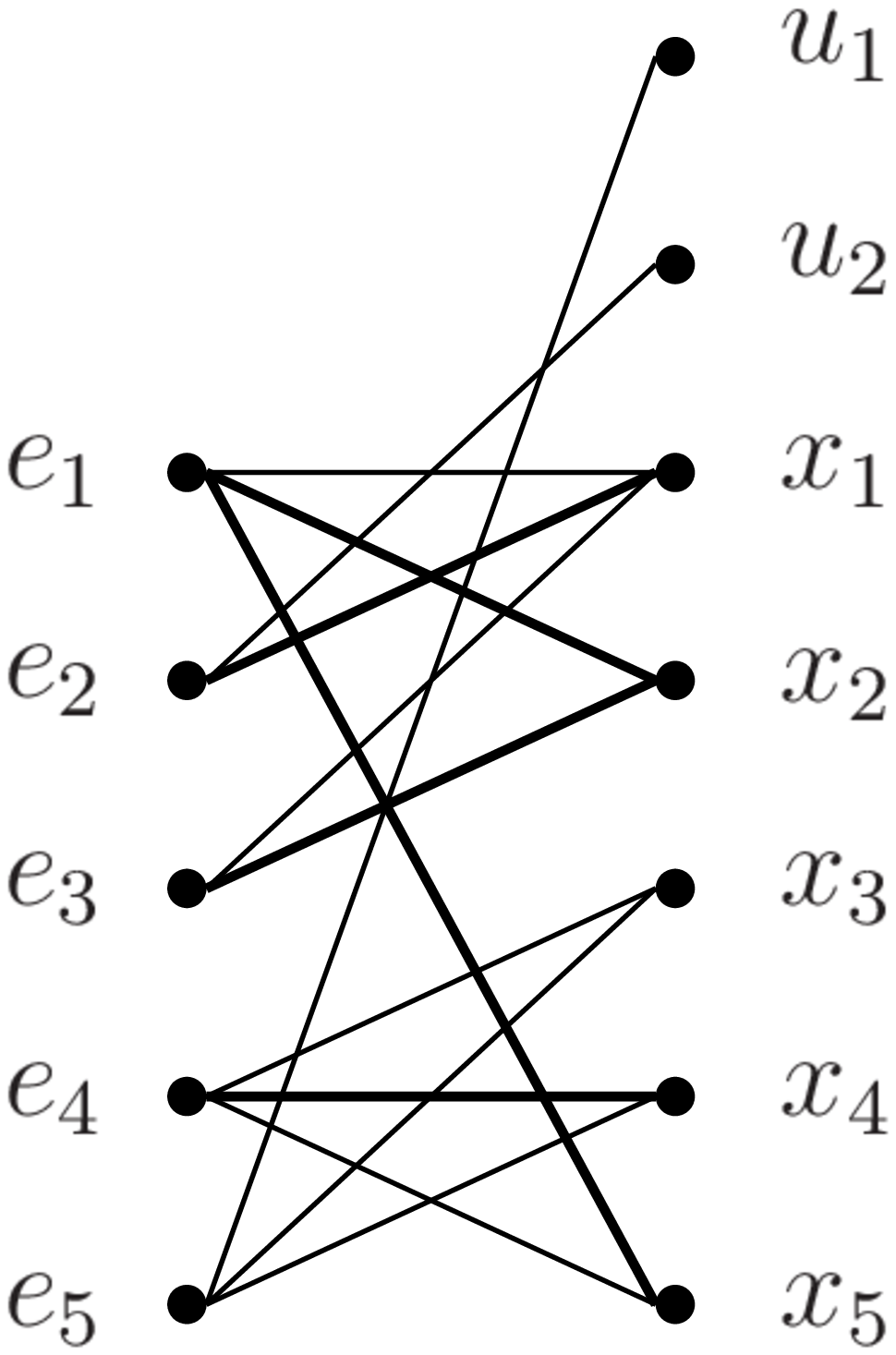}
       \subcaption{$G$}
     \end{minipage} 
   \end{tabular}
    \caption{Bipartite graph representation of descriptor system (\ref{eq:descriptor}) with (\ref{eq:examplesystem}) and $B_2$ in (\ref{eq:exampleinputs}). The bold edges represent s-arcs.}
    \label{fig:example}
\end{figure*}      
    
\begin{figure}
    \centering
    \includegraphics[width=3.5cm]{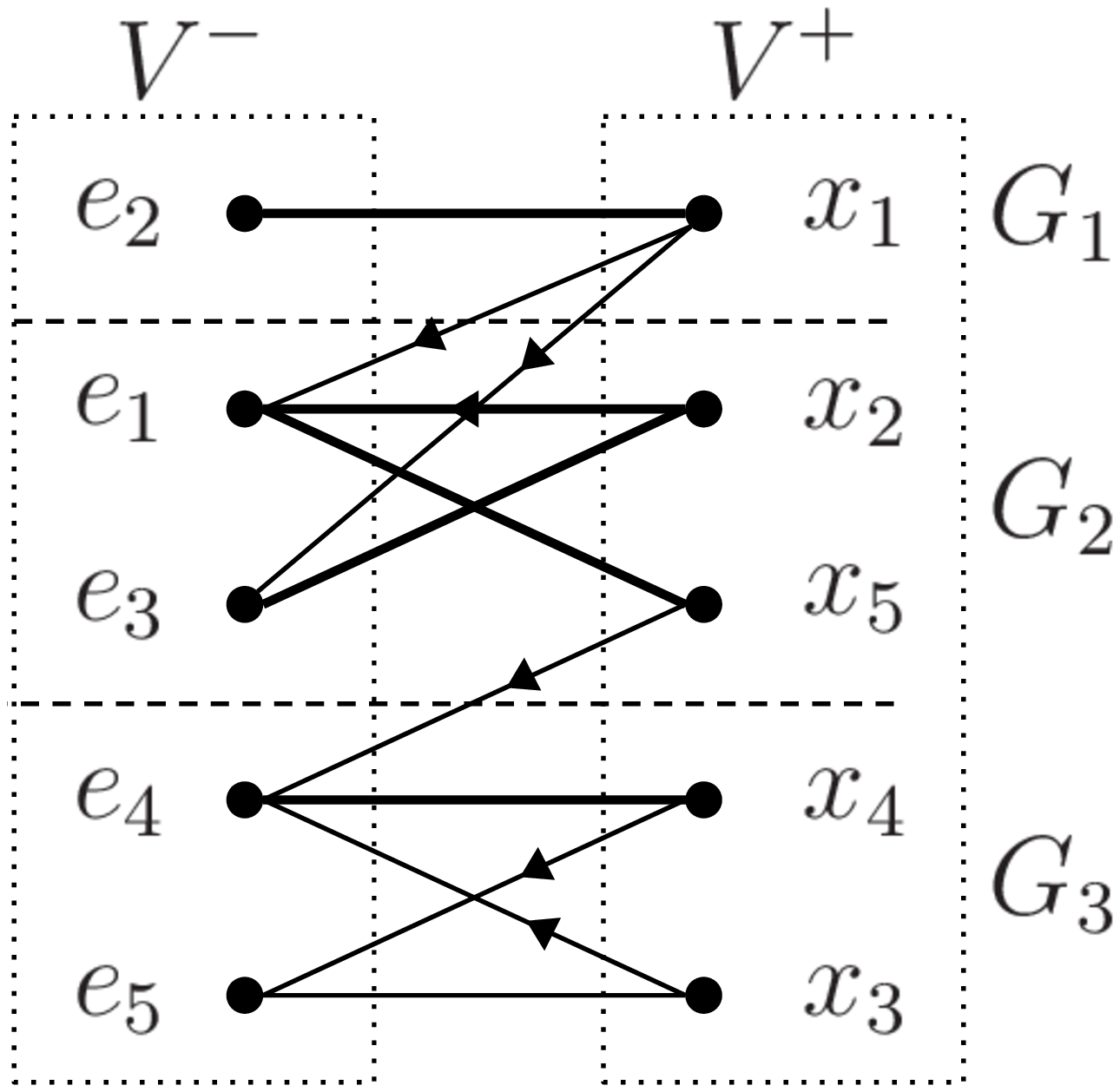}
    \caption{DM decomposition of $G_{A-sF}$ of descriptor system (\ref{eq:descriptor}) with (\ref{eq:examplesystem}). The bold edges represent s-arcs.}
    \label{fig:exampledm}
\end{figure}
\begin{figure}
    \centering
    \includegraphics[width=3.5cm]{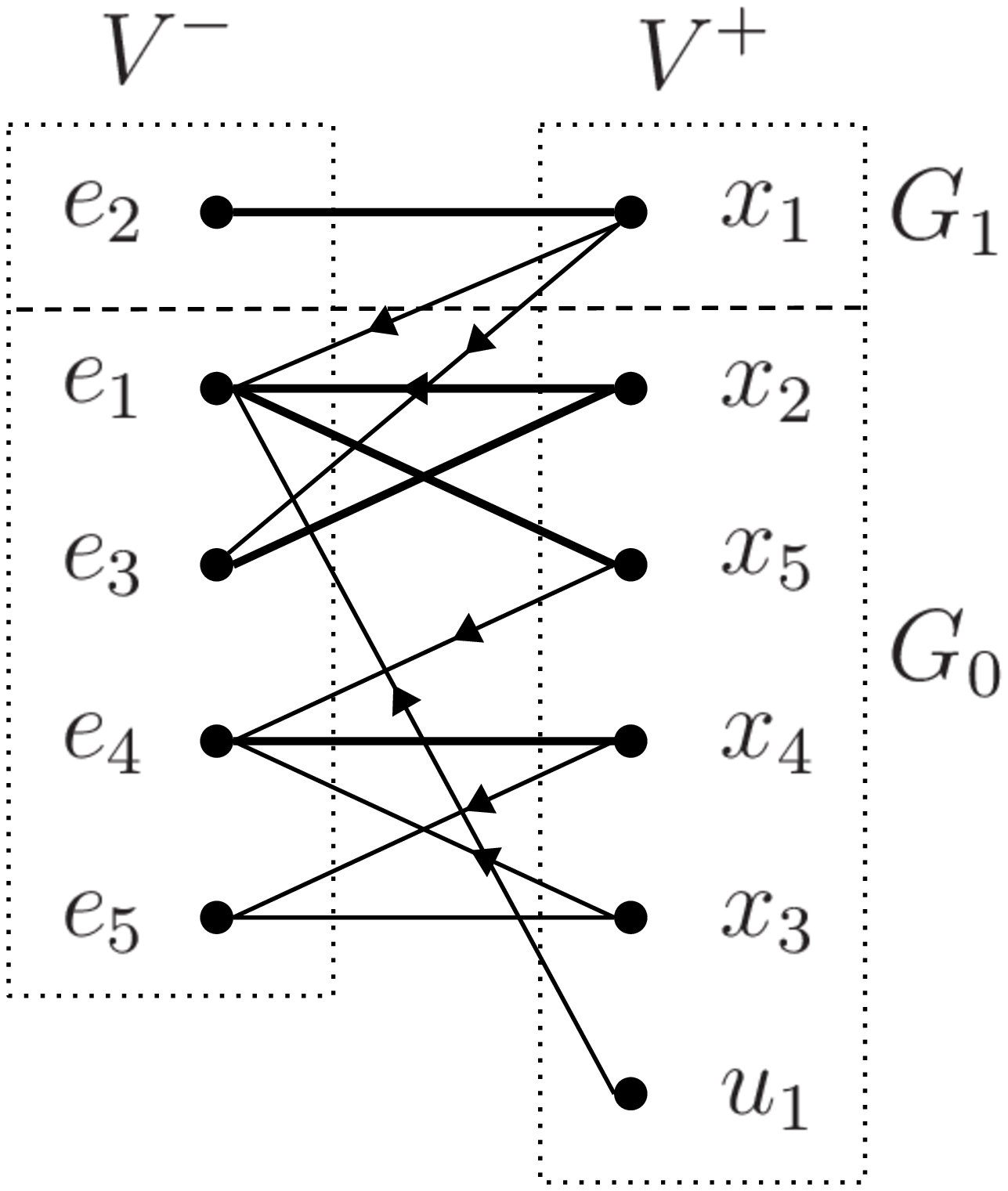}
    \caption{DM decomposition of $G$ of descriptor system (\ref{eq:descriptor}) with (\ref{eq:examplesystem}) and $B_1$ in (\ref{eq:exampleinputs}). The bold edges represent s-arcs.}
    \label{fig:exampledmwithinput}
\end{figure}

\begin{figure}
    \centering
    \includegraphics[width=3.5cm]{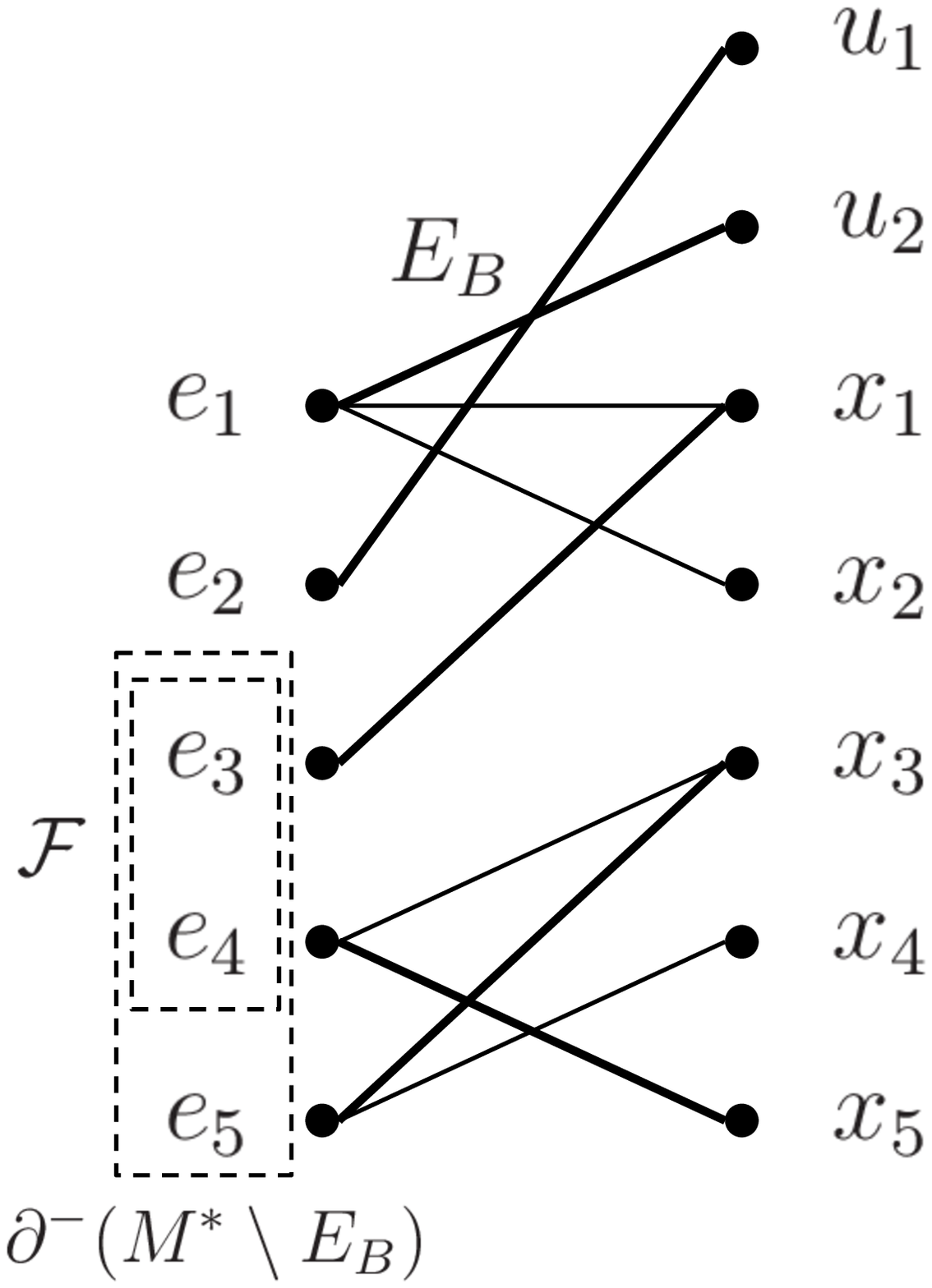}
    \caption{The optimal solution $B$ in (\ref{eq:solutionB}) to Problem (\ref{prob:cmcp0}) with $\mathcal{F} = \{e_3, e_4\}$. The bold edges represent the maximum matching $M^\ast$ of $G_{\left[A\mid B\right]}$. The optimality of $B$ will be shown in Section \ref{subsec:algforprobcmcp0}.}
    \label{fig:thm1_proofb}
\end{figure}
 
Furthermore, for a DM component of $\mathcal{G}$ that has s-arcs \cite{terasaki2021minimal}, we call it a DM s-component, where $\mathcal{G}$ is $G$ or $G_{A-sF}$.
Let $G_k \ ( k = 0 , \dots , b , \infty )$ be DM components of $\mathcal{G}$, and $G_k \ ( k = 1,\dots,b )$ are called consistent DM components; $G_0$ and $G_\infty$ are termed inconsistent DM components. 
The maximal consistent DM s-component is a consistent DM s-component $G_k$ of $\mathcal{G}$ such that no other DM s-components are greater than $G_k$ related to the partial order $\preceq$. Note that there can be multiple maximal consistent DM s-components.

For instance, consider descriptor system (\ref{eq:descriptor}) with parameters given in (\ref{eq:examplesystem}).
The corresponding DM decomposition of the bipartite representation $G_{A-sF}$ is depicted in Fig.~\ref{fig:exampledm}. $G_1$, $G_2$, and $G_3$ are the consistent DM components of $G_{A-sF}$, and $G_3 \preceq G_2 \preceq G_1$.
That is, a maximal consistent DM s-component of this graph is $G_1$, because there is no edge to enter $G_1$ from other DM components.

The following graph-theoretic characterization of structural controllability for descriptor system (\ref{eq:descriptor}) is found in \cite{murota00}.
\begin{proposition}
\label{prop:sc}
Descriptor system (\ref{eq:descriptor}) is structurally controllable \textit{if and only if}
\begin{enumerate}
    \item $\nu(G_{A-sF}) = n$,
    \item $\nu(G_{\left[A \mid B\right]}) = n$,
    \item No consistent DM components of $G$ contain s-arcs,
\end{enumerate}
where $\nu(\mathcal{G})$ is the size of a maximum matching of a bipartite graph $\mathcal{G}$.
\end{proposition}

Condition 1) represents the solvability of the system described in Eq.~(\ref{eq:solvability}) in Section~\ref{sec:prb},
and the latter two conditions correspond to (\ref{eq:rank}) in Proposition \ref{prop:scalg}.
From assumption (\ref{eq:solvability}), condition 1) in Proposition~\ref{prop:sc} is always satisfied.
Assumption (\ref{eq:solvability}) also implies that there are no inconsistent DM components $G_0$ and $G_\infty$ of $G_{A-sF}$.
This is because if we choose a maximum matching $M$ with a size of $n$ in $G_{A-sF}$, then $V_0$ and $V_\infty$ in Steps 3--4 in Algorithm \ref{alg:dmdecomp} are both empty.

The following lemma describes the characterization of the consistent DM components of $G_{A-sF}$ and $G$.
\begin{lemma}
    \label{lem:inconsistent}
    Consider descriptor system (\ref{eq:descriptor}) which satisfies solvability condition (\ref{eq:solvability}). Then, the following conditions are equivalent\textcolor{black}{:}
    \begin{enumerate}
        \item A node $v \in X\cup V^-$ of a consistent DM component $G_i$ in $G_{A-sF}$ belongs to an inconsitent DM component $G_0$ in $G$.
        \item There is a directed path on $\tilde{G}_M$ from an input node $u\in U$ to $v \in X\cup V^-$ in Step 2 of Algorithm \ref{alg:dmdecomp} for $G$.
    \end{enumerate}
\end{lemma}
\begin{proof}
    From assumption (\ref{eq:solvability}), condition 1) in Proposition~\ref{prop:sc} holds.
    This means that we can find the same maximum matching $M$ with size of $n$ in Step 1 of Algorithm \ref{alg:dmdecomp} for $G_{A-sF}$ and $G$.
    Then, in Step 2 of Algorithm \ref{alg:dmdecomp}, $\tilde{G}_M$ of $G_{A-sF}$ is a directed subgraph of $\tilde{G}_M$ of $G$, as illustrated in Fig.~\ref{fig:dmdecompalgorithm}.
    Also, $V^+ = X\cup U$, $\partial^+ M = X$, $V^- = \partial^- M$. Thus, $V^+\setminus \partial^+ M$ and $V^-\setminus \partial^- M$ in Steps 3 and 4 of Algorithm \ref{alg:dmdecomp} coincide with $U$ and $\emptyset$, respectively. 
    That is, the node sets $V_0$ and $V_\infty$ of the inconsistent DM components in $G$ are as follows:
    \begin{align*}
        V_0 = U\cup \{v \in X\cup V^- \mid \exists u \in U\ u \to_{\tilde{G}_M} v \}, \ V_\infty = \emptyset.
    \end{align*}
    This implies that 1) and 2) are equivalent. \qed
\end{proof}

For instance, consider descriptor system (\ref{eq:descriptor}) with (\ref{eq:examplesystem}) and $B_1$ in (\ref{eq:exampleinputs}). The DM decomposition of the bipartite representation $G$ is depicted in Fig.~\ref{fig:exampledmwithinput}.
$G_0$ is the inconsistent DM component, $G_1$ is the consistent DM component, and $G_0 \preceq G_1$. That is, $G_1$ is the maximal consistent DM s-component.
Then, the consistent DM component $G_1$ in $G_{A-sF}$ (Fig~\ref{fig:exampledm}) is also consistent in $G$ (Fig.~\ref{fig:exampledmwithinput}), since there is no directed path from $u_1$ to $e_2$ or $x_1$ in $G$.
Note that in general, there is a directed path from $u\in U$ to $x\in X$,
because edges in a matching are undirected.

Lemma \ref{lem:inconsistent} gives another characterization of the structural controllability condition for system (\ref{eq:descriptor}) which satisfies solvability condition (\ref{eq:solvability}). This property will be used in the proof of Theorem \ref{thm:existence} in the next subsection.
\begin{corollary}
    \label{cor:sc2}
    Consider descriptor system (\ref{eq:descriptor}) which satisfies solvability condition (\ref{eq:solvability}). Then, descriptor system (\ref{eq:descriptor}) is structurally controllable \textit{if and only if}
    \begin{itemize}
        \item[2')] $\nu(G_{\left[A \mid B\right]}) = n$,
        \item[3')] For each maximal consistent DM s-component $G_i$ in $G_{A-sF}$, which is a subgraph of $G$, there is a directed path on $\tilde{G}_M$ from $U$ to $G_i$ in Step 2 of Algorithm \ref{alg:dmdecomp} for $G$.
    \end{itemize}
\end{corollary}

\begin{proof}
    From assumption (\ref{eq:solvability}), condition 1) in Proposition~\ref{prop:sc} holds.
    Condition 2') is equivalent to condition 2) in Proposition \ref{prop:sc}.
    Also, condition 3) in Proposition \ref{prop:sc} holds \textit{if and only if} all nodes $v\in X\cup V^-$ of consistent DM s-components $G_i$ in $G_{A-sF}$ belong to an inconsistent DM s-component $G_0$ in $G$, since consistent DM s-components contain s-arcs.
    Using Lemma \ref{lem:inconsistent}, this is equivalent to the following:
    \begin{itemize}
        \item[3*)] For each consistent DM s-component $G_i$ in $G_{A-sF}$, which is a subgraph of $G$, there is a directed path on $\tilde{G}_M$ from $U$ to $G_i$.
    \end{itemize}
    This condition is also equivalent to 3') in Corollary \ref{cor:sc2}, since the non-maximal consistent DM components in $G_{A-sF}$ have a directed path from the maximal consistent DM component in $G_{A-sF}$ from the definition of the partial order $\preceq$. \qed
\end{proof}

Consider descriptor system (\ref{eq:descriptor}) with \textcolor{black}{$A$ and $F$ given in} (\ref{eq:examplesystem}) and $B_1$ in (\ref{eq:exampleinputs}) again.
In $G_{A-sF}$, there is no directed path from $u_1$ to $G_1$ (Fig.~\ref{fig:exampledmwithinput}). 
This means condition 3') in Corollary \ref{cor:sc2} does not hold.
Thus, descriptor system (\ref{eq:descriptor}) with (\ref{eq:examplesystem}) and $B_1$ in (\ref{eq:exampleinputs}) is not structurally controllable.

\subsection{Existence of solution to Problem (\ref{prob:cmcp0})}
Unlike the traditional MCP0, it is not obvious whether or not an optimal solution exists for MCP0 with forbidden equations.
By using Proposition \ref{prop:sc} and Corollary \ref{cor:sc2}, we obtain the following theorem which \textcolor{black}{characterizes} the existence of solutions to Problem (\ref{prob:cmcp0}).
This theorem is employed to construct the optimal solution to Problem (\ref{prob:cmcp0}) in Section \ref{subsec:algforprobcmcp0}.
\begin{theorem}
\label{thm:existence}
\textcolor{black}{
Problem (\ref{prob:cmcp0}) has a solution if and only if both of the following conditions hold simultaneously:}
\begin{enumerate}
    \renewcommand{\labelenumi}{\alph{enumi}).}
    \item For each node set $V_i$ of maximal consistent DM s-components $G_i$ in $G_{A-sF}$, there exists $e\in V_i \cap V^-$ such that $e\not\in \mathcal{F}$.
    \item The graph-theoretic problem
    \begin{align}
        \label{prob:forbidden_matching}
        \begin{cases}
            \underset{\text{$M$}}{\text{maximize}} & |M| \\
            \text{subject to} & \text{$M$ is matching of $G_{A}$,} \\
            & \mathcal{F} \subseteq \partial^{-}M
        \end{cases}
    \end{align}
    is feasible.
\end{enumerate}
\end{theorem}
\begin{proof}
    We assume that $B\in \mathcal{G}^{n\times m}$ is \textcolor{black}{a} solution to Problem (\ref{prob:cmcp0}),
    \textcolor{black}{and prove that conditions a) and b) hold.
    Note that} this $B$ defines $E_B$ in (\ref{eq:defgraph}).
    Suppose that a maximal consistent DM s-component $G_i$ of $G_{A-sF}$, which is a subgraph of $G$, is not connected to $U$\textcolor{black}{,
    where $U$ denotes the set of input nodes.}
    This means that there is no directed path on $\tilde{G}_M$ for $G$ from $U$ to $G_i$ in Step 2 of Algorithm \ref{alg:dmdecomp}. 
    From Corollary \ref{cor:sc2}, this implies that system (\ref{eq:descriptor}) with \textcolor{black}{the} given $B$ is not structurally controllable, which contradicts the assumption.
    Thus, in $G_{A-sF}$, for the all node set $V_i$ of maximal consistent DM s-components $G_i$, there exists an equation node $e\in V^-\cap V_i$ which is connected to an input $u\in U$. This means that $e\not \in \mathcal{F}$ and condition a) holds.
    \textcolor{black}{Moreover,} 
     condition 2) in Proposition \ref{prop:sc} \textcolor{black}{implies that} the maximum matching $M^\ast$ of $G_{[A\mid B]} = (V^+,V^-\cup U; E\cup E_B)$ and $|M^\ast| = n$ exists.
    Then, $M^\ast\setminus E_B$ is a matching in $G_{A}$ and $\partial^- (M^\ast\setminus E_B)$ contains $\mathcal{F}$ since $\partial^- E_B \subseteq V^-\setminus\mathcal{F}$ (Fig.~\ref{fig:thm1_proofb}).
    Thus, $M^\ast\setminus E_B$ is a feasible solution to Problem (\ref{prob:forbidden_matching})\textcolor{black}{, and} condition b) holds.

    \textcolor{black}{Next, we assume that conditions a) and b) hold, and prove that Problem (\ref{prob:cmcp0}) has a solution.}
    Consider the input nodes $U$ with a size of $n-|M^\ast|$, where $M^\ast$ is the optimal solution to Problem (\ref{prob:forbidden_matching}).
    Then we can connect each equation node $e\in V^-\setminus \partial^- M^\ast \subseteq V^-\setminus \mathcal{F}$ to an input node.
    This means that the edge set $E_B$ in (\ref{eq:defgraph}) with size $|E_B| = n-|M^\ast|$ was constructed and forms a part of a matching of $G_{\left[A\mid B\right]}$.
    That is, $M^\ast\cup E_B$ is a maximum matching of $G_{\left[A\mid B\right]}$, and thus condition 2') in Corollary \ref{cor:sc2} holds. 
    Also, for each maximal consistent DM s-component $V_i$ of $G_{A-sF}$, we can connect an equation node $e\in V_i\setminus \mathcal{F}$ to an input from condition a).
    That is, all maximal consistent DM s-components of $G_{A-sF}$ have directed paths from the input node set $U$.
    This means that condition 3') in Corollary \ref{cor:sc2} holds.
    Thus, Problem (\ref{prob:cmcp0}) has a solution. \qed
    \end{proof}
    
    Theorem \ref{thm:existence} implies that the existence of an optimal solution to MCP0 with forbidden equations for system (\ref{eq:descriptor}) can be verified using pure graph theory.

    For instance, consider descriptor system (\ref{eq:descriptor}) with (\ref{eq:examplesystem}) depicted in Figs.~\ref{fig:example} and \ref{fig:exampledm}.
    \begin{enumerate}
        \item Consider the case $\mathcal{F} = \{e_1, e_3\}$. Then,
        \textcolor{black}{MCP0 with $\mathcal{F}$ is not feasible since condition b) in Theorem \ref{thm:existence} does not hold.}
        {In fact,}
        the maximal consistent DM s-component $G_1$ in $G_{A-sF}$ has a node $e_2$ which does not belong to $\mathcal{F}$.
    This means that condition a) in Theorem \ref{thm:existence} holds.
    However, since a set of nodes that connect to $e_1$ or $e_3$ is $\{x_1\}$ in $G_A$ in Fig.~\ref{fig:example}-(a), there is no matching $M$ of $G_A$ which satisfies $\mathcal{F}\subseteq \partial^-M$. Thus, Problem (\ref{prob:forbidden_matching}) is not feasible.
    
        \item Consider the case $\mathcal{F} = \{e_2\}$.
    Then, 
    \textcolor{black}{MCP0 with $\mathcal{F}$ is not feasible since condition a) does not hold. In fact,}
    we can choose a feasible solution for Problem \eqref{prob:forbidden_matching} as $M_{\mathcal{F}} := \{(e_2, x_1)\}$ which satisfies $\partial^- M_{\mathcal{F}} = \{e_2\}\supseteq \mathcal{F}$. Thus, condition b) holds. 
    However, a node of a maximal consistent DM s-component $G_1$ in $G_{A-sF}$ is only $e_2$ and $e_2\in \mathcal{F}$.
    \end{enumerate}
    \subsection{Optimal value for Problem (\ref{prob:cmcp0})}
    Using the solution to Problem (\ref{prob:forbidden_matching}), we can compute the optimal value for Problem (\ref{prob:cmcp0}).
    \begin{theorem}
        \label{thm:cmcp0}
        Suppose that an optimal solution to Problem (\ref{prob:cmcp0}) exists. If the optimal value of Problem (\ref{prob:forbidden_matching}) is $m^\ast$, then the minimum number of inputs that satisfy the constraints of Problem (\ref{prob:cmcp0}) is 
        \begin{align}
            \label{eq:mininput}
            n_D = \max\{n-m^\ast, 1\}.
        \end{align}
    \end{theorem}
    \begin{proof}
    \textcolor{black}{We assume that} $B$ is an optimal solution \textcolor{black}{with
     $n_D$ column size} to Problem (\ref{prob:cmcp0}),
     \textcolor{black}{and prove that
     $n_D \geq n - m^\ast$, where
     $m^*\leq n$ by the definition.
     To this end, we assume}
        $n_D < n-m^\ast$. 
    From condition 2) in Proposition \ref{prop:sc}, $\nu(G_{\left[A \mid B\right]}) = n$ holds.
    Let $M$ be a maximum matching of $G_{\left[A\mid B\right]}$. Then, $M' := M\setminus E_B$ is a feasible solution to Problem (\ref{prob:forbidden_matching}) and
        $|M'| \geq n - n_D$. 
    Thus, we obtain $|M'| > m^\ast$. This contradicts the maximality of $m^\ast$.

    \textcolor{black}{If $n_D=0$, the corresponding system of the form \eqref{eq:descriptor} does not satisfy condition I) of Problem (\ref{prob:cmcp0}).
    Thus, $n_D \geq \max\{n-m^\ast, 1\}$.}
    
    \textcolor{black}{To prove that the equality holds, that is, 
    \eqref{eq:mininput} holds,
    we assume that} $M_{\mathcal{F}}^\ast$ is an optimal solution to Problem (\ref{prob:forbidden_matching}) \textcolor{black}{such that $|M_{\mathcal{F}}^\ast|=m^\ast$, and construct a system that satisfies the constraints of Problem (\ref{prob:cmcp0}) as follows.}
\begin{enumerate}
    \item Connect each input node to a node of $V^{-} \setminus \partial^{-} M_{\mathcal{F}}^\ast$. \textcolor{black}{Then,} $\nu(G_{\left[A \mid B \right]}) = n$ is satisfied and \textit{i.e.}, \textcolor{black}{condition 2') in Corollary~\ref{cor:sc2}} holds. 

    \item Connect the input nodes to \textcolor{black}{all} maximal consistent DM s-components of $G_{A-sF}$. Then \textcolor{black}{condition 3') in Corollary~\ref{cor:sc2}} is satisfied without increasing the number of input nodes.  
\end{enumerate}

    \textcolor{black}{A system constructed by 1) and 2) satisfies condition I) of Problem \eqref{prob:cmcp0}.
    Moreover, a system constructed by 1)  satisfies condition II), which is equivalent to that $\partial^- E_B \subseteq V^- \setminus \mathcal{F}$.
    This is because $\mathcal{F} \subseteq \partial^- M_{\mathcal{F}}^\ast$.
    Note that in 2), we can choose nodes so that condition II) is satisfied, because
    condition a) in Theorem \ref{thm:existence} implies that all maximal consistent DM s-component has at least one node that is not in $\mathcal{F}$.
    } 
    Therefore, this system \textcolor{black}{of the form} \eqref{eq:descriptor} \textcolor{black}{satisfies the constraints of Problem \eqref{prob:cmcp0} and $n_D$ is given by \eqref{eq:mininput}.}
    \qed
\end{proof}

If the set $\mathcal{F}$ of forbidden equations is empty, Eq. (\ref{eq:mininput}) can be written as $n_D = \max\{n-\nu(G_A), 1\}$.
This result is consistent with the result of the standard MCP0 for system (\ref{eq:descriptor}) \cite{terasaki2021minimal}.

The argument in the proof of Theorem \ref{thm:cmcp0} will be used to develop an algorithm for Problem (\ref{prob:cmcp0}) in Section \ref{subsec:algforprobcmcp0}.

\subsection{Algorithm for Problem (\ref{prob:forbidden_matching})}

Algorithm \ref{alg:forbidden_matching} describes an algorithm for solving Problem (\ref{prob:forbidden_matching}) based on an alternating path algorithm \cite{korte2012combinatorial}.
Steps 1--4 check the feasibility of solving Problem (\ref{prob:forbidden_matching}).
This is because if there is a feasible solution $M$ to Problem (\ref{prob:forbidden_matching}), by restricting the matching nodes in $\partial^- M$ to $\mathcal{F}\cap \partial^- M$, a matching $M_{\mathcal{F}}$ of $G_{\mathcal{F}} = (V^+,\mathcal{F}; E_A)$ is obtained that satisfies $|\mathcal{F}|=|M_{\mathcal{F}}|$ (Fig.~\ref{fig:alg2proof}). 
The converse is shown to hold by the following alternating path algorithm in Steps 5--9, as shown in the proof of Theorem \ref{thm:algvalid}.

\begin{figure}
    \centering
    \includegraphics[width=7.5cm]{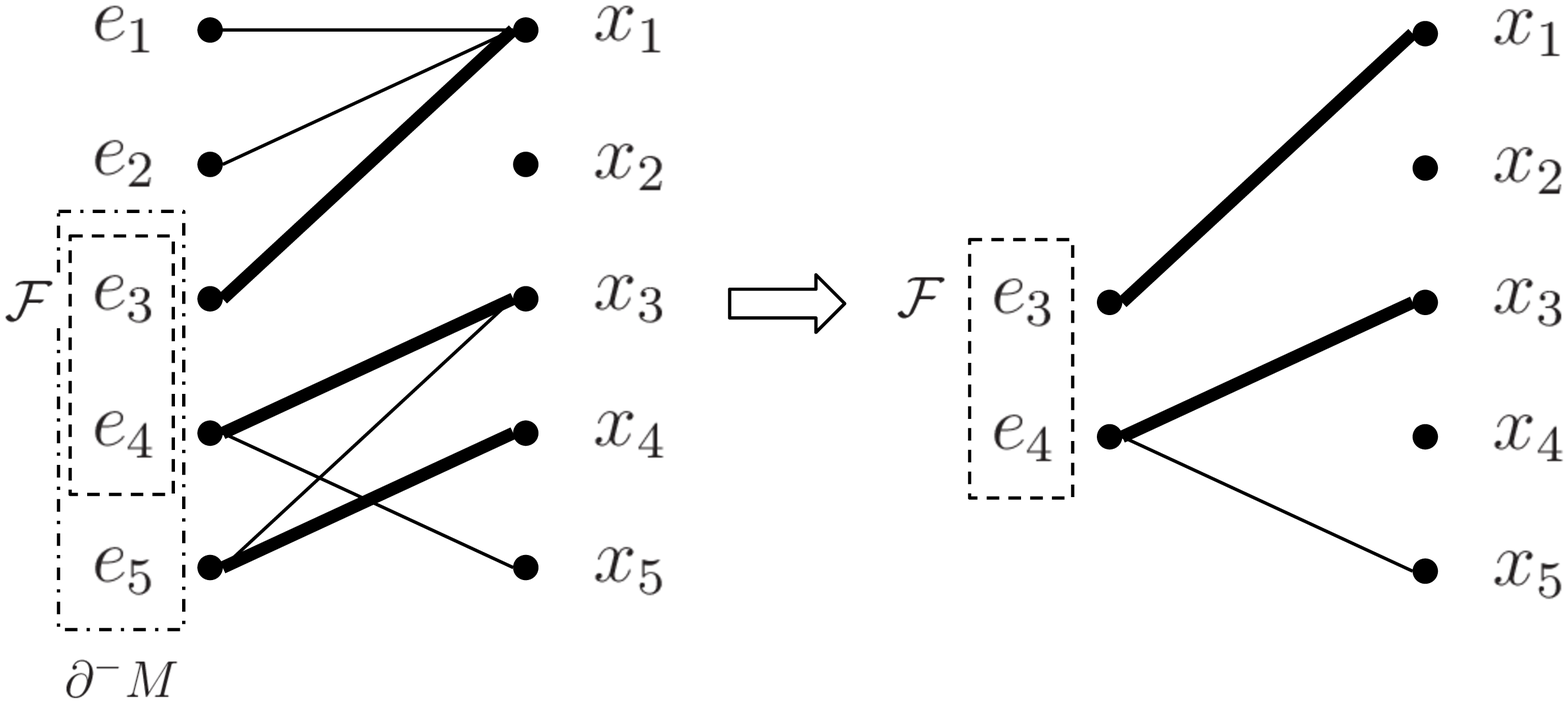}
    \caption{$G_A$ and $G_{\mathcal{F}}=(V^+, \mathcal{F}; E_A)$. The bold edges on the left are a feasible solution $M$ to Problem (\ref{prob:forbidden_matching}) while the right is a matching $M_{\mathcal{F}}$ of $G_{\mathcal{F}}$ which satisfies $|\mathcal{F}| = |M_{\mathcal{F}}|$.}
    \label{fig:alg2proof}
\end{figure}

The alternating path algorithm that is used in Steps 5--9 is explained as follows: Let $M^\ast$ be a maximum matching of $G_A$ which does not satisfy the condition $\mathcal{F}\subseteq \partial^-M^\ast$ in (\ref{prob:forbidden_matching}). That is, $M^\ast$ is not a feasible solution to Problem (\ref{prob:forbidden_matching}).
Note that $M^\ast$ can be computed by the Hopcraft-Karp algorithm \cite{korte2012combinatorial}.
Then, $v^-\in \mathcal{F}\setminus\partial^{-}M^\ast$ exists. Suppose a path 
\begin{align*}
    P & = p_1 \not\in M^\ast \to p_2 \in M^\ast \to\cdots \\
      & \to p_{2l-1}\not\in M^\ast \to p_{2l} \in M^\ast, \quad l=1,2,\dots,
\end{align*}
exists, which starts with $v^-$ and ends with $v_t^-\in V^-\setminus\mathcal{F}$.
Then, \textcolor{black}{a new maximum matching $M'$ is defined by the symmetric difference $(M^*\setminus P)\cup (P\setminus M^*)$ between $M^\ast$ and $P$.}
Moreover, the number of $\mathcal{F}\setminus\partial^{-}M'$ is just one less than the number of $\mathcal{F}\setminus\partial^{-}M^\ast$ since $\partial^{-}M' = (\partial^{-}M^\ast\setminus\{v_t^-\})\cup\{v^-\}$. This procedure is illustrated in Fig.~\ref{fig:alternatingpath2}.
By repeating this process, we can finally have a maximum matching $M_{\mathcal{F}}^\ast$ of $G_A$ which satisfies $\mathcal{F}\subseteq \partial^{-}M_{\mathcal{F}}^\ast$. 
Note that \textcolor{black}{
    from the construction of the alternating path, no paths that start at different vertices $v^-\in\mathcal{F}\setminus\partial^-M^\ast$ share edges with each other. This means that all desired alternating paths can be found by visiting each edge only once. Thus, we can execute Steps 6--9 by the breadth-first search algorithm with time complexity of $O(|E|+|V|)$ \cite{korte2012combinatorial}.
}
\begin{figure}
    \centering
    \includegraphics[width=7.7cm]{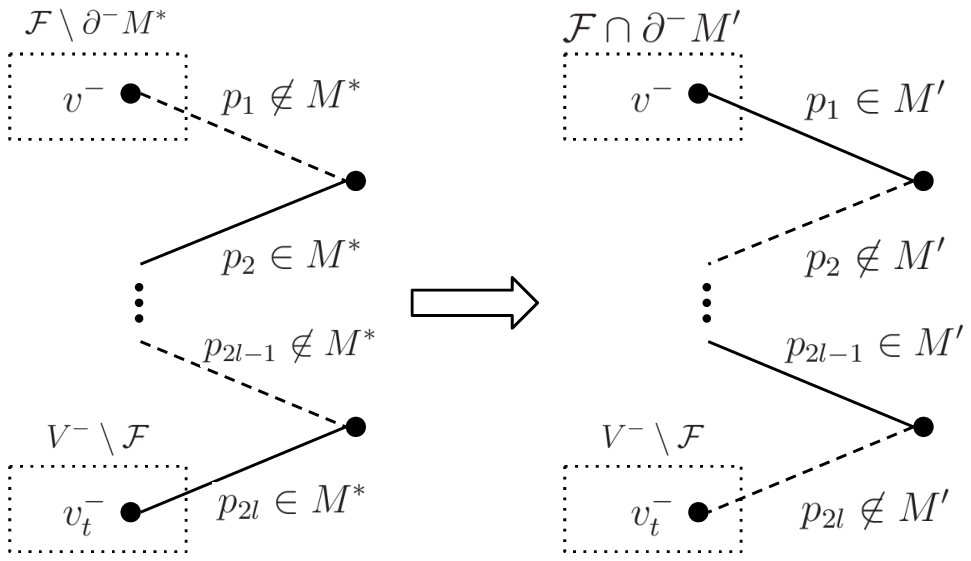}
    \caption{Alternating path algorithm.}
    \label{fig:alternatingpath2}
\end{figure}

\begin{figure}[!t]
  \begin{algorithm}[H]
    \caption{Algorithm for solving Problem (\ref{prob:forbidden_matching})}
    \label{alg:forbidden_matching}
    \begin{algorithmic}[1]
        \State Find the maximum matching $M^\ast_\mathcal{F}$ of the graph $G_\mathcal{F} := (V^+, \mathcal{F}; E_A)$.
        \If{$|\mathcal{F}| > |M^\ast_\mathcal{F}|$}
            \State Problem (\ref{prob:forbidden_matching}) is infeasible.
        \EndIf
        \State Find the maximum matching $M^\ast$ of $G_A$.
        \For{$v^-\in \mathcal{F}\setminus\partial^{-}M^\ast$}
            \State Find an alternating path \textcolor{black}{$P=\{p_1, \cdots, p_{2l}\}$} which starts with $v^-$ and ends with $v_t^{-}\in V^{-}\setminus \mathcal{F}$
            \State \textcolor{black}{$M^\ast := M^\ast\cup\{p_1,p_3,\cdots,p_{2l-1}\}\setminus \{p_2,p_4,\cdots,p_{2l}\}$}     
        \EndFor
        \State Output $M^\ast$.
    \end{algorithmic}
  \end{algorithm}
\end{figure}

\begin{figure}[!t]
  \begin{algorithm}[H]
    \caption{Algorithm for solving Problem (\ref{prob:cmcp0})}
    \label{alg:cmcp0}
    \begin{algorithmic}[1]
        \If{condition a) in Theorem \ref{thm:existence} is not satisfied}
            \State Problem (\ref{prob:cmcp0}) is infeasible.
        \EndIf
        \State Run Algorithm \ref{alg:forbidden_matching} and let $M^\ast$ be its output that is an optimal solution to Problem (\ref{prob:forbidden_matching}).
        \State Connect inputs $U := \{u_1,\dots, u_{n_D}\}$ to each equation node of $V^-\setminus \partial^- M^\ast$, where $n_D$ is (\ref{eq:mininput}) in Theorem \ref{thm:cmcp0}.
        \State Connect inputs $U$ to an arbitrary node of each maximal consistent DM s-components of $G_{A-sF}$, which does not belong to $\mathcal{F}$.
    \end{algorithmic}
  \end{algorithm}
\end{figure}

\begin{theorem}
    \label{thm:algvalid}
    If Problem (\ref{prob:forbidden_matching}) is feasible, Algorithm \ref{alg:forbidden_matching} outputs the optimal solution of Problem (\ref{prob:forbidden_matching}).
\end{theorem}
\begin{proof}
    Let $M^\ast$ be a maximum matching of $G_A$.
    From the feasibility assumption of Problem (\ref{prob:forbidden_matching}), there exists a maximum matching $M_{\mathcal{F}}^\ast$ of the bipartite graph $(V^+,V^-\cap \mathcal{F}; E_A)$ by Steps 1--4 in Algorithm \ref{alg:forbidden_matching}.
    Because of the construction of Algorithm \ref{alg:forbidden_matching}, it is sufficient to show that there is an alternating path $P$ for each $v^-\in\mathcal{F}\setminus\partial^-M^\ast$. 
    \textcolor{black}{
    In particular, we show that such $P$ alternates between edges of $M^\ast$ and $M_{\mathcal{F}}^\ast\setminus M^\ast$.
    }
    \textcolor{black}{
    From the definition of $M_{\mathcal{F}}^\ast$,} an edge $p\in M_{\mathcal{F}}^\ast$ exists that connects to $v^-$. 
    \textcolor{black}{
    Also, there exists $p'\in M^\ast$ that $\partial^+p' = \partial^+ p$. This is because otherwise $M^\ast \cup\{p\}$ is a matching, whose size is larger than $M^\ast$ and this contradicts the maximality of $M^\ast$.
    Thus, there is a path with alternating the edges of $M^\ast$ and $M_{\mathcal{F}}^\ast\setminus M^\ast$ with a length of at least 2.
    }
    To show that all alternating paths end with a node $v_t^-$ of $V^-$, not of $V^+$, 
    suppose that an alternating path $P$ exists that ends with $\tilde{v}^+\in V^+$.
    We also assume that $P$ alternates between the edges of $M^\ast$ and \textcolor{black}{$M_{\mathcal{F}}^\ast\setminus M^\ast$}.
    In this case, we have $|M_{\mathcal{F}}^\ast \cap P| = |M^\ast\cap P|+1$, 
    since $P$ must start with $p_1\in M_{\mathcal{F}}^\ast$ and end with $p_{2l-1}\in M_{\mathcal{F}}^\ast$ (Fig. \ref{fig:alternatingpath}).
    Thus, by replacing the $M^\ast \cap P$ edges with $M_{\mathcal{F}}^\ast \cap P$ edges, we can obtain a matching larger than $M^\ast$ in size. This contradicts the maximality of $M^\ast$. Thus, $p_{2l}\in M^\ast$ exists and $P\cup\{p_{2l}\}$ is an alternating path.
    That is, $p_{2l}$ connects $v_t^-\in V^-$.
    \textcolor{black}{
    If $v_t^-\in V^-\cap \mathcal{F}$, there is an edge $p' \in M_{\mathcal{F}}^\ast\setminus M^\ast$ that satisfies $v_t^- = \partial^- p'$, and $P\cup\{p_{2l}\}\cup \{p'\}$ is a new alternating path between the edges of $M^\ast$ and $M_{\mathcal{F}}^\ast\setminus M^\ast$.
    Thus, by inductively applying the above argument, we finally have an alternating path $P$ which ends with a node of $V^-\setminus\mathcal{F}$.
    }
    \qed
\end{proof}
\begin{figure}
    \centering
    \includegraphics[width=3.5cm]{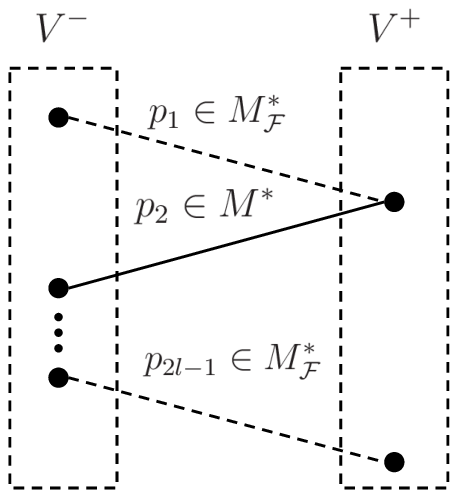}
    \caption{\textcolor{black}{Alternating path $P = \{p_1,p_2\dots, p_{2l-1}\}$. $p_1, p_3, \dots, p_{2l-1}$ are in $M_{\mathcal{F}}^\ast$, and $p_2, p_4, \dots, p_{2l-2}$ are in $M^\ast$.}}
    \label{fig:alternatingpath}
\end{figure}

%
\textcolor{black}{
To explain how Algorithm \ref{alg:forbidden_matching} works, consider Problem \eqref{prob:forbidden_matching} with $G_A$ of descriptor system (\ref{eq:examplesystem}) and forbidden equations $\mathcal{F} = \{e_3, e_4\}$.
We can choose $M_{\mathcal{F}}^\ast$ in Step 1 of Algorithm \ref{alg:forbidden_matching} as $M_{\mathcal{F}}^\ast = \{(e_3, x_1), (e_4, x_3)\}$.
In Step 5, we have the maximum matching $M^\ast = \{(e_3, x_1), (e_4, x_3), (e_5, x_4)\}$.
Moreover, Step 6--9 illustrated in Fig.~\ref{fig:example_alternatingpath} produces the new maximum matching $M' = \{(e_3, x_1), (e_4, x_3), (e_5, x_4)\}$ that satisfies $\mathcal{F}\subseteq \partial^- M' = \{e_3, e_4, e_5\}$ in \eqref{prob:forbidden_matching}.
In fact, we can find an alternating path $P = \{(e_1, x_1), (e_3, x_1)\}$ and obtain the new matching $M'$.
}

\begin{figure}
    \centering
    \includegraphics[width=8cm]{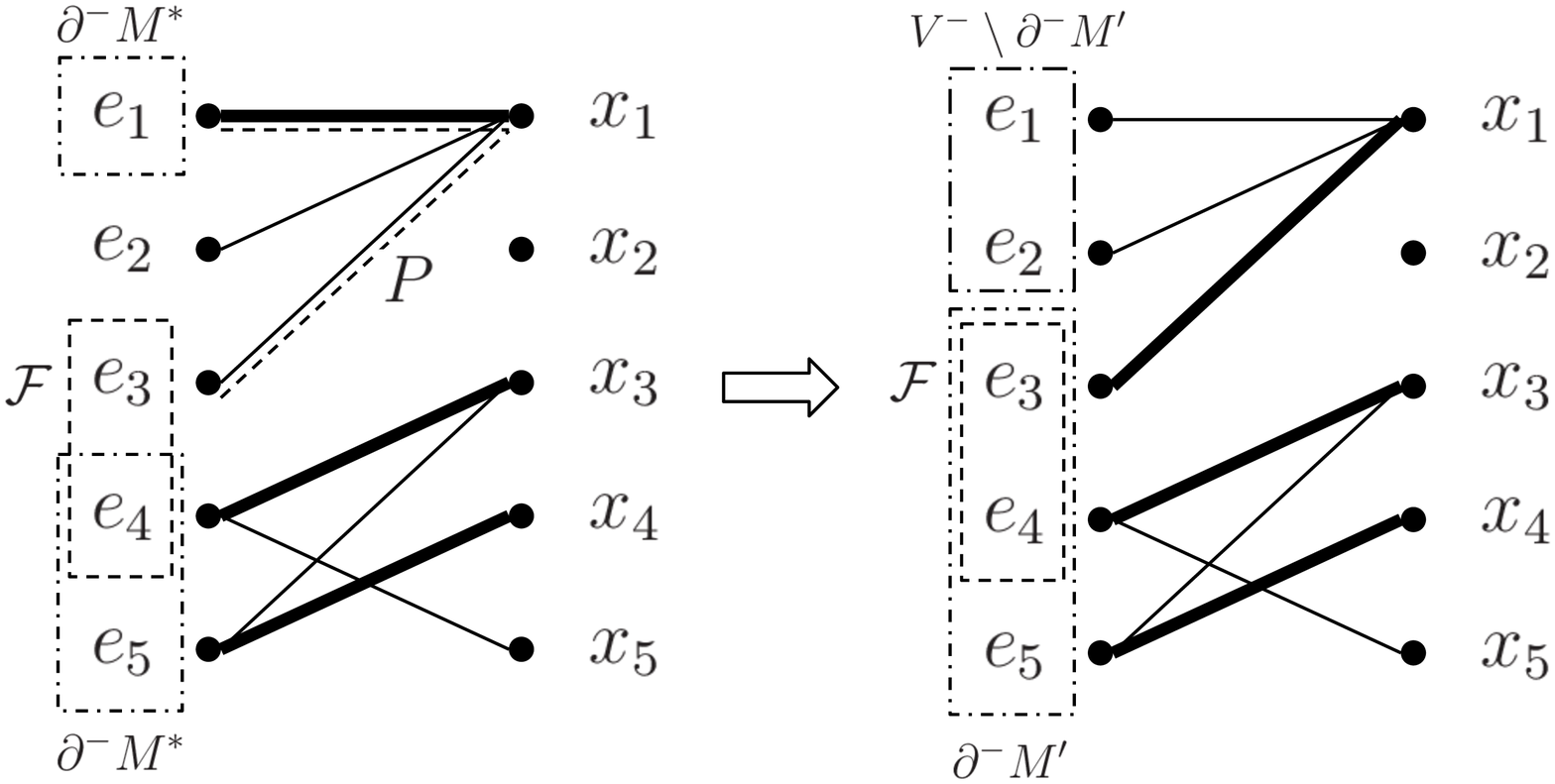}
    \caption{\textcolor{black}{Example of Step 6--9 in Algorithm \ref{alg:forbidden_matching}}. The bold edges represent $M^\ast$ (left) and $M'$ (right).}
    \label{fig:example_alternatingpath}
\end{figure}

\subsection{Algorithm for Problem (\ref{prob:cmcp0})}
\label{subsec:algforprobcmcp0}
Algorithm \ref{alg:cmcp0} shows an algorithm for solving Problem (\ref{prob:cmcp0}) and has the following property.

\begin{theorem}
    \label{thm:compcmcp0}
    \textcolor{black}{If Problem (\ref{prob:cmcp0}) has a feasible solution, then}
    Algorithm \ref{alg:cmcp0} outputs the optimal solution to Problem (\ref{prob:cmcp0}). 
    \textcolor{black}{If Problem (\ref{prob:cmcp0}) is infeasible, Algorithm \ref{alg:cmcp0} determines the infeasibility.}
    Furthermore, the time complexity is $O(|V|+|E|\sqrt{|V|})$.
\end{theorem}

\begin{proof}
Algorithm \ref{alg:cmcp0} is based on the argument in the proof of Theorem \ref{thm:cmcp0}.
Steps 1--4 determine whether or not Problem (\ref{prob:cmcp0}) is feasible, and Steps 1--3 are computed by Algorithm \ref{alg:dmdecomp}. 
Step 4 checks condition b) in Theorem \ref{thm:existence} and output an optimal solution $M^\ast$ of Problem (\ref{prob:forbidden_matching}).
\textcolor{black}{Steps 5 and 6 in Algorithm \ref{alg:cmcp0} output the optimal solution to Problem \eqref{prob:cmcp0}, as shown in the proof of Theorem \ref{thm:cmcp0}.}

\textcolor{black}{
    Next, we show that the time complexity is $O(|V|+|E|\sqrt{|V|})$.
    Step 1--3 in Algorithm \ref{alg:cmcp0} can be checked by Algorithm \ref{alg:dmdecomp}, its time complexity is $O(|E|\sqrt{|V|})$.
    The time complexity of Step 4 is $O(|V|+|E|\sqrt{|V|})$.
    In fact, in Steps 1--5 in Algorithm \ref{alg:forbidden_matching} are computed by using the Hopcroft-Karp algorithm\cite{korte2012combinatorial}, the time complexity is $O(|E|\sqrt{|V|})$. Moreover, Steps 6--9 in Algorithm \ref{alg:forbidden_matching} can be computed by breadth-first search algorithm in $O(|V|+|E|)$, as mentioned already.
    Steps 5--6 in Algorithm \ref{alg:cmcp0} can be computed in $O(|V|)$ at most.
    Therefore, the time complexity of Algorithm \ref{alg:forbidden_matching} is $O(|V|+|E|\sqrt{|V|})$. \qed
}
\end{proof}

\textcolor{black}{
Theorem \ref{thm:compcmcp0} means that more general problems than those of \cite{liu2011controllability, terasaki2021minimal}  can be solved with the same computational complexity.
In fact,
the problem addressed by \cite{liu2011controllability} is a special case of Problem \eqref{prob:cmcp0} with $\mathcal{F}=\emptyset$ and $F=I_n$, and it can be solved in $O(|V|+|E|\sqrt{|V|})$.
Moreover, the algorithm proposed in \cite{terasaki2021minimal} deals with Problem \eqref{prob:cmcp0} with $\mathcal{F} = \emptyset$, and the time complexity  is $O(|V|+|E|\sqrt{|V|})$.
}

To illustrate how Algorithm \ref{alg:cmcp0} works, consider descriptor system (\ref{eq:descriptor}) with (\ref{eq:examplesystem}) and forbidden equations $\mathcal{F} = \{e_3, e_4\}$.
In this case, Steps 1--4 determine that Problem (\ref{prob:cmcp0}) is feasible. In fact, the maximal consistent DM s-component $G_1$ in Fig.~\ref{fig:exampledm} has a node $e_2$ which does not belong to $\mathcal{F}$.
\textcolor{black}{
Also, in Step 4, we obtain $M^\ast = \{(e_3, x_1), (e_4, x_3), (e_5, x_4)\}$ from the discussion of how Algorithm \ref{alg:forbidden_matching} works in IV-C.
Thus, Problem (\ref{prob:cmcp0}) is feasible, since conditions a) and b) in Theorem \ref{thm:existence} hold. 
}
Furthermore, Steps 5 and 6 produce
\begin{align}
    \label{eq:solutionB}
    B = \begin{bmatrix}
        b_2 & 0 & 0& 0& 0\\
        0 & b_1& 0& 0& 0
    \end{bmatrix}^\top,
\end{align}
that is an optimal solution to Problem (\ref{prob:cmcp0}).
In fact, from Theorem \ref{thm:cmcp0}, we have the minimum number of inputs $n_D = n-|M^\ast| = 2$. Also, $V^-\setminus \partial^- M' = \{e_1, e_2\}$. Thus, connecting $u_1$ to $e_1$, and $u_2$ to $e_2$ , we have $B$ as (\ref{eq:solutionB}).


\section{Conclusion} \label{sec:conclusion}

In this study, we introduced the forbidden equations to MCP0 for structural descriptor systems.
We \textcolor{black}{gave} a necessary and sufficient condition for the existence of solutions to the problem and provided the solution to MCP0. The algorithm for solving the problem can be computed in polynomial time as for the standard MCP0.
That is, our proposed algorithm \textcolor{black}{is well positioned for applications} to large-scale descriptor systems.

In this paper, we focused on MCP0 for a structural descriptor system, since, for a structural descriptor system, MCP1 is NP-hard in general as shown in \cite{terasaki2021minimal}. Thus, finding a solvable condition for MCP1 with forbidden equations in polynomial time would be a future project.

\textcolor{black}{
Furthermore, structural controllability considered in this paper requires that the system parameters be algebraically independent. This means that all non-zero system parameters are free, which may be a strong assumption for practical situations. To avoid this assumption, strong structural controllability has been proposed in \cite{mayeda1979strong},
and studied from a graph-theoretic perspective in \cite{jia2021unifying}.
A strong structural controllability problem version in this paper 
is one of the future works.
}


%


\section*{Acknowledgment}
This work was supported by the Japan Society for the Promotion of Science KAKENHI under Grant 20K14760 and 23K03899. 

\ifCLASSOPTIONcaptionsoff
  \newpage
\fi



\bibliographystyle{IEEEtran}
\bibliography{main.bib}




%




\end{document}